\documentclass[12pt, reqno,oneside]{amsart}
\usepackage[utf8]{inputenc}

\usepackage{amssymb, amsmath, amsthm, wasysym, mathrsfs}
\usepackage{hyperref}
\usepackage[alphabetic,lite]{amsrefs}
\usepackage{fullpage}
\usepackage{setspace}
\usepackage{mathtools}
\usepackage{enumitem}

\usepackage{ bbold }

\usepackage[usenames,dvipsnames]{color}
\newtheorem{lemma}{Lemma}[section]
\newtheorem{lemma*}{Lemma}
\newtheorem{theorem}[lemma]{Theorem}

\newtheorem{question}[lemma]{Question}
\newtheorem{claim*}{Claim}

\theoremstyle{definition}

\newcommand{\Z}{{\mathbb Z}}

\numberwithin{equation}{section}
\numberwithin{table}{section}
\title{Characterizing Candidates for Cannon's Conjecture from Geometric Measure Theory}
\author{Tamunonye Cheetham-West, Alexander Nolte$^1$ \\ {Rice University}}
\thanks{\hspace{-0.5cm}$^1$This material is based upon work supported by the National Science Foundation under Grant No. 1842494 and Grant No. 2005551}
\begin{document}
\bibliographystyle{plain}

\maketitle
\begin{abstract}
    We show that recent work of Song implies that torsion-free hyperbolic groups with Gromov boundary $S^2$ are realized as fundamental groups of closed $3$-manifolds of constant negative curvature if and only if the solution to an associated spherical Plateau problem for group homology is isometric to such a $3$-manifold, and suggest some related questions.
\end{abstract}

\section{Introduction}

Cannon's conjecture \cite{cannon1998conjecture} states that a Gromov-hyperbolic group $\Gamma$ with $\partial \Gamma \cong S^2$ that acts effectively on its boundary is a Kleinian group. For torsion-free groups $\Gamma$, this is equivalent to the existence of a closed orientable hyperbolic $3$-manifold $Y$ so that $\pi_1(Y) \cong \Gamma$. 

Cannon's conjecture was first suggested as part of a program to prove Thurston's hyperbolization conjecture \cite{Thurston1982Geometrization}, but remains open despite Perelman's proof of the Geometrization Theorem \cite{perelman2002entropy}, \cite{perelman2003extinction}, \cite{perelman2003surgery}. A large body of work surrounds Cannon's conjecture including, for instance, an equivalent condition in terms of $\varepsilon$-Fuchsian subgroups due to Markovi\'c \cite{Markovic2013criterion}, a formulation in terms of the Ahlfors regular conformal dimension of the Gromov boundary by Bonk and Kleiner \cite{bonkkleiner2005}, and a resolution of analogous questions in some larger dimensions due to Bartels, L\"uck, and Weinberger \cite{bartels2010hyperbolic}.

In this brief note, we present an equivalent formulation of Cannon's conjecture for torsion-free groups $\Gamma$ (Theorem \ref{headline-result}) in terms of the structure and regularity of solutions to a minimization problem studied in recent work of Song \cite{song2022plateau}. In this view, the challenge of Cannon's conjecture is changed from producing a distinguished action on $\mathbb{H}^3$ from a vast family to an analysis of specific candidate geometric spaces that are characterizing in the following sense. If the candidate space for a particular torsion-free group $\Gamma$ is a solution to Cannon's conjecture for $\Gamma$, then by Mostow Rigidity it is the unique such solution up to isometry; if it is not, then no solution to Cannon's conjecture exists for $\Gamma$.

We work in the framework of the spherical Plateau problem, as described in \cite{song2022plateau} (see also Section \ref{spherical-plateau-summary}). Specifically, given a torsion-free countable group $\Gamma$, a Hilbert manifold classifying space $\mathcal{S}(\Gamma)/\Gamma$ for $\Gamma$ is constructed with Hilbert-Riemannian metric $g_{\text{Hil}}$. The group homology $H_*(\Gamma, \mathbb{Z})$ is then identified with the singular homology of $\mathcal{S}(\Gamma)/\Gamma$. For $h \in H_k(\Gamma, \mathbb{Z})$, a solution to the spherical Plateau problem for $h$ is an integral current space $(X_\infty,d_\infty, S_\infty)$ that is the intrinsic-flat limit (see \cite{sormani2011intrinsic}) of a sequence $C_i$ of $k$-dimensional integral currents in the sense of Ambrosio and Kirchheim \cite{ambrosio2000currents} in $\mathcal{S}(\Gamma)/\Gamma$ representing $h$ whose mass approaches the infimum of the masses of such currents. The immediately relevant feature of integral current spaces here is that $(X_\infty, d_\infty)$ is a complete metric space. Solutions to the spherical Plateau problem always exist as a consequence of Wenger compactness \cite{wenger2011compactness}.

The observation of this note is that the results in \cite{song2022plateau} and work of Bestvina and Mess \cite{bestvina1991groups} imply that solutions to appropriate spherical Plateau problems provide distinguished candidates for Cannon's conjecture, constructed using only the group $\Gamma$. This is in line with the theme (e.g. \cite{meeks1982plateau}, \cite{perelman2003surgery}, \cite{hass2983slice}) that geometric-analytic methods can often furnish fruitful candidates in existence questions in geometry and topology.

\begin{theorem}\label{headline-result}
    Let $\Gamma$ be a torsion-free hyperbolic group with $\partial \Gamma \cong S^2$ that acts effectively on its boundary, let $(X_\infty, d_\infty, S_\infty)$ be any solution of the spherical Plateau problem for a generator $h_\Gamma$ of $H_3(\Gamma, \mathbb{Z}) \cong \mathbb{Z}$, and let $d^L_\infty$ be the path metric of $d_\infty$. There exists a closed orientable hyperbolic $3$-manifold with fundamental group isomorphic to $\Gamma$ if and only if $(X_\infty,\sqrt{3} d^L_\infty)$ is isometric to a closed orientable hyperbolic $3$-manifold $(Y, g_{{\text{\rm{hyp}}}})$ with $\pi_1(Y) \cong \Gamma$.
\end{theorem}

By Theorem \ref{headline-result}, to disprove Cannon's conjecture it would suffice to exhibit the presence of a singularity, failure of compactness, or incompatibility of fundamental group in an appropriate spherical Plateau problem solution. On the other hand, the Hyperbolization Theorem implies these are the only ways Cannon's conjecture can fail: to prove Cannon's conjecture via the approach suggested by Theorem \ref{headline-result} it suffices to show that the Plateau problem solutions $X_\infty$ are closed topological $3$-manifolds with $\pi_1(X_\infty) \cong \Gamma$.

One expects that such an approach would need to navigate substantial challenges: in analogous finite-dimensional Plateau problems, solutions can have singularities of codimension at least two \cite{almgren1983codimension2}. So an affirmative resolution to Cannon's conjecture through Theorem \ref{headline-result} would require stronger regularity results in this setting than are expected in the general case.

  We remark that Theorem 0.4 in \cite{song2022plateau} ensures that the Plateau problem solutions appearing in Theorem \ref{headline-result} are at least nonempty, which rules out the most dramatic possible failure of Cannon's conjecture from this perspective.
 \subsection{Some Questions} We conclude the introduction with some open questions on the spherical Plateau problem, selected with an eye towards developing tools to approach Cannon's conjecture and related topics. A range of further questions appears in \cite{song2022plateau}.
 
 From the perspective of Theorem \ref{headline-result}, the regularity of solutions to the spherical Plateau problem plays a central role in Cannon's conjecture. The development of an appropriate theory analogous to that in finite dimensions (e.g. \cite{almgren1983codimension2}) is expected to be a substantial undertaking: in the finite-dimensional setting, known proofs of regularity results are technical and involved. See \cite{ambrosio2016survey} for a survey of the finite-dimensional case, and \cite{ambrosio2018hilbertreg} for an example of the state of the art in regularity questions for locally mass-minimizing integral currents in infinite dimensions.
 
 In the present case, the existing theory is contained in \cite{song2022plateau}, and addresses questions of (non)-triviality for Plateau problem solutions under conditions on the group $\Gamma$. It is not yet clear what sorts of singularities to expect in spherical Plateau problem solutions for homology classes of hyperbolic groups. So a natural first question is:
 
\begin{question}\label{singularity-examples}
 Study examples of singularity formation in spherical Plateau problem solutions. Develop models for singularities that occur in the spherical Plateau problem for homology classes of hyperbolic groups.
\end{question}

For instance, it would be interesting to understand the possibilities for singularity formation within the class of spherical Dehn fillings studied by Song in \cite{song2022plateau}. A long-term goal would be to classify possible singularities in restricted families of examples.

In a different direction, by Theorem \ref{headline-result} and work of Markovi\'c \cite{Markovic2013criterion}, one class of groups for which spherical Plateau problem solutions are closed hyperbolic $3$-manifolds are groups that satisfy the hypotheses of Cannon's conjecture and contain sufficiently many quasi-convex surface subgroups. Here, ``sufficiently many'' means that for all distinct $ p, q \in \partial \Gamma$, there is a quasi-convex surface subgroup $H$ so that $p$ and $q$ lie in different connected components of $\partial \Gamma - \partial H$. With the hope of developing methods in setting of the spherical Plateau problem, it would be interesting to

\begin{question}
   Prove Markovi\'c's criterion for Cannon's conjecture \cite{Markovic2013criterion} in the framework of the spherical Plateau problem.
\end{question}
 We now turn to questions less directly concerned with regularity. Song shows in \cite{song2022plateau} (Theorem 3.3) that results in \cite{besson1995rigidites} and \cite{besson1996rigidites} imply that after a homothety by a fixed factor, the spherical volume of the fundamental class of a closed connected oriented hyperbolic $n$-manifold agrees with the volume of the manifold. To assess potential counterexamples to Cannon's conjecture, a direction of investigation that seems productive is:
 
 \begin{question}
  Develop methods of placing bounds on spherical volumes of Plateau problem solutions. Produce algorithms that allow for computational tools to find explicit bounds on spherical volumes of examples. 
 \end{question}
 
 One expects for upper bounds to be constructed through building examples, and for demonstration of lower bounds to be more challenging. We remark that one potential, though unlikely-appearing, application of spherical volume bounds is that were an appropriate group homology class's Plateau problem solution to have spherical volume beneath that of the Weeks manifold, results in \cite{gabai2009minimum} would disprove Cannon's conjecture.

Our next two questions address structural features of solutions to the spherical Plateau problem in our setting that are not present in finite-dimensional analogues. To frame our discussion, recall that for a homology class $h \in H_k(M,\mathbb{Z})$, where $(M,g)$ is a Riemannian manifold, a solution to the Plateau problem for $h$ is a mass-minimizing integral $k$-current $T$ in $M$ representing $h$. In our setting, a solution $C_\infty$ to the spherical Plateau problem for $h \in H_k(\Gamma, \mathbb{Z})$ is an integral current \textit{space} that is an \textit{intrinsic}-flat limit of integral currents $C_i$ in $\mathcal{S}(\Gamma)/\Gamma$.

The intrinsic structural questions about Plateau problem solutions studied in the finite-dimensional setting, most prominently about regularity, have direct analogues in our setting as discussed above. Our focus in the next two questions is instead on qualitative structural differences introduced by the change in frameworks.

The first difference is that a solution $C_\infty$ to the spherical Plateau problem is only a limit of integral currents $C_i$ in $\mathcal{S}(\Gamma)/\Gamma$ in an intrinsic sense, and need not be realized by an integral current in $\mathcal{S}(\Gamma)/\Gamma$. In particular, $C_\infty$ is does not need to represent a homology class in a fixed space, in contrast to the finite-dimensional case. This gives rise to some new topological flexibility of $C_\infty$ in our setting.

There are two aspects of the approach to Cannon's conjecture suggested by Theorem \ref{headline-result} related to this flexibility. They seem likely to be able to be addressed independently of the smoothness of Plateau problem solutions. 
 
 \begin{question}\label{partial-steps-question}
  Let $\Gamma$ be a torsion-free hyperbolic group with Gromov boundary $S^2$ and let $h \in H_3(\Gamma, \mathbb{Z})$ be a generator with associated spherical Plateau problem solution $(X_\infty, d_\infty, S_\infty)$. Denote the associated path-metric by $d_\infty^L$.
  \begin{enumerate}[label=(\alph*)]
      \item \label{question-closed} Suppose that the metric space $(X_\infty, \sqrt{3} d_\infty^L)$ is isometric to an orientable hyperbolic 3-manifold $(Y, g_{\text{hyp}})$. Must $Y$ be closed?
      \item \label{question-fundamental-group} Suppose further that $Y$ is closed. Is $\pi_1(Y) \cong \Gamma$?
  \end{enumerate}
 \end{question}
 
  A cause for interest in Question \ref{partial-steps-question} is that affirmative answers would represent partial progress towards a proof of Cannon's conjecture through the perspective of Theorem \ref{headline-result}, and a negative answer to either point would give a disproof of Cannon's conjecture.
  
  We remark here that if $\Gamma$ is the fundamental group of a closed $3$-manifold $Y$ whose geometrization has a hyperbolic part with at least two connected components, the analogue of Statement \ref{question-closed} of Question \ref{partial-steps-question} for the fundamental class $[[Y]]$ of $Y$ is false as a consequence of work of Song (\cite{song2022plateau}, Theorems 3.5-3.6). Instead, any spherical Plateau problem solution $(X_\infty, d_\infty, T_\infty)$ for $[[Y]]$ has $X_\infty$ homeomorphic to the disjoint union of the hyperbolic parts $\{Y_k\}$ of $Y$. On each component $Y_k$, the induced path-metrics of the unique finite-volume hyperbolic metric $g_\text{hyp}^k$ and $d_\infty$ are homothetic. For manifolds $Y$ as above, these metrics have cusps. So $X_\infty$ is not closed.

 The second structural phenomenon in the spherical Plateau problem we consider is that for any infinite torsion-free group $\Gamma$, the classifying space $\mathcal{S}(\Gamma)/\Gamma$ has points of arbitrarily small injectivity radius. The complications introduced by this are made precise by the notion of ``collapsed" and ``non-collapsed'' parts of solutions $C_\infty$ to the spherical Plateau problem. We sketch their definition below; for a full account, see \cite{song2022plateau} Section 5.
 
Let $C_\infty$ be a solution to the spherical Plateau problem for $h \in H_n(\Gamma, \mathbb{Z})$. Song shows that there is a sequence $C_i$ of integral currents in $\mathcal{S}(\Gamma)/\Gamma$ approximating $C_\infty$ that are ``pulled tight'' in the following sense. The underlying integral current spaces of $C_i$ and $C_\infty$ can be realized as integral currents in a common Banach space $Z$ so that $C_i \to C_\infty$ in the flat topology on currents in $Z$, and for any $k \in \mathbb{N}$ the support of the part $C^{\geq 1/k}$ of $C_i$ supported in the $1/k$-thick part of $\mathcal{S}(\Gamma)/\Gamma$ converges in the Hausdorff topology in $Z$ to a compact subset $X_{k,\infty}$ of the support of $C_\infty$. 

The restriction $C_\infty^{> 0}$ of $C_\infty$ to $\bigcup_{k = 1}^\infty X_{k,\infty}$ is called a \textit{non-collapsed part} of $C_\infty$, and the restriction of $C_\infty$ to the complement of $\bigcup_{k = 1}^\infty X_{k,\infty}$ is called a \textit{collapsed part} of $C_\infty$. Here, $C_\infty^{> 0}$ depends on the choice of pulled-tight sequence $C_i$, the Banach space $Z$, and the realizations of $C_i$ and $C_\infty$ inside $Z$.

In this way, a non-collapsed part $C_\infty^{> 0}$ of $C_\infty$ is the union of parts of $C_\infty$ that can be investigated without concern over arbitrarily thin regions of $\mathcal{S}(\Gamma)/\Gamma$. It is expected that the analysis of non-collapsed parts of spherical Plateau problem solutions will be more straightforward than that of collapsed parts (e.g. \cite{song2022plateau} Theorem 5.6). For this reason, we ask:
 
\begin{question}\label{noncollapse-question}
 If $\Gamma$ is a torsion-free hyperbolic group with $\partial \Gamma \cong S^2$, can it always be arranged for a spherical Plateau problem solution for a generator of $H_3(\Gamma,\mathbb{Z})$ to have trivial collapsed part?
\end{question}

We note that the hypotheses of Question \ref{noncollapse-question} are stronger than that of Theorem 0.4 in \cite{song2022plateau}, which ensures that a solution to the spherical Plateau problem can be taken to at least have non-empty non-collapsed part. 

 Another remark in a similar direction to the spirit of Question \ref{partial-steps-question} is that even though the spaces $(X_\infty, \sqrt{3} d_\infty^L)$ appearing in Theorem \ref{headline-result} are unique up to isometry if $\Gamma$ is Kleinian, this is not known to be the case if Cannon's conjecture fails for $\Gamma$. So the presence of non-path-isometric integral current spaces solving the spherical Plateau problem under the hypotheses of Theorem \ref{headline-result} would invalidate Cannon's conjecture. In contrast to the matters considered in Question \ref{partial-steps-question}, though, it is not necessary to prove uniqueness of path-isometry classes of these solutions to the spherical Plateau problem to establish Cannon's conjecture through the perspective suggested by Theorem \ref{headline-result}. This motivates the
 
 \begin{question}\label{uniqueness-question}
    Let $\Gamma$ and $h$ be as in Question \ref{partial-steps-question}, and let $(X_{\infty}, d_{\infty}, S_{\infty}), (Z_\infty, \delta_\infty, T_\infty)$ be two solutions to the spherical Plateau problem for $h$ with path-metrics $d_\infty^L, \delta_\infty^L$ respectively. Are $(X_\infty ,d_\infty^L)$ and $(Z_\infty,\delta_\infty^L)$ isometric? 
 \end{question}
 

  The only known technique for proving uniqueness features of solutions to the spherical Plateau problem is to examine barycenter maps between approximate solutions and model spaces (e.g. \cite{song2022plateau}, proofs of Theorems 3.4, 3.6, 4.3). This approach is limited by the requirement of prior knowledge of well-behaved model spaces.
  
  In the absence of a distinguished model space, uniqueness may well fail. For instance, if $S$ is a closed, orientable surface of genus $g \geq 2$, $\Gamma = \pi_1(S)$, and $h$ generates $H_2(\Gamma, \mathbb{Z}) \cong \mathbb{Z}$, then Poisson embeddings produce all hyperbolic metrics on $S$ (after homothety), and appropriate limits, among the isometry types of solutions to the spherical Plateau problem for $h$ (see \cite{song2022plateau}, Section 3).
  
  Distinguished models do not appear to be available in our setting, though the flexibility of hyperbolic metrics in dimension $2$ that gives rise to non-uniqueness in the above example is also absent. So to positively resolve Question \ref{uniqueness-question} without proving Cannon's conjecture would require a new method.
 
 Finally, one direction of inquiry communicated to the authors by Song \cite{song2022email} is to determine to what extent the spherical Plateau problem solutions for a group $\Gamma$ inform the algebraic structure of $\Gamma$. For instance, one question related to Question \ref{partial-steps-question}, Statement \ref{question-fundamental-group} in this direction is:
 \begin{question}
  Suppose that $\Gamma$ is a countable group and $h \in H_n(\Gamma, \mathbb{Z})$ has a path-connected spherical Plateau problem solution $(X_\infty, d_\infty, S_\infty)$ with path-metric $d_\infty^L$ isometric to the path-metric of a closed Riemannian $n$-manifold $(M, g)$ of constant negative sectional curvature. Must $\Gamma$ contain a subgroup isomorphic to $\pi_1(M)$?
 \end{question}
 
  \vspace{0.25cm}
 \par \noindent \textbf{Acknowledgements.} This note exists thanks to an inspiring talk of Antoine Song and subsequent conversations at the conference Recent Developments on Geometric Measure Theory and its Applications (on the occasion of the retirement of Bob Hardt). A.N. thanks Mike Wolf and T.C-W thanks Alan Reid for their support, and the authors thank Christos Mantoulidis, Chris Leininger, and Antoine Song for thoughtful comments on previous drafts of this note.
\section{The Spherical Plateau Problem}\label{spherical-plateau-summary}

We briefly review the spherical Plateau problem in the torsion-free setting, as defined in \cite{song2022plateau}, where proofs and detailed definitions pertaining to the following can be found.

Let $\Gamma$ be a countably infinite torsion-free discrete group, $\mathbf{H}$ a separable Hilbert space, and denote the unit sphere in $\ell^2(\Gamma, \mathbf{H})$ by $\mathcal{S}(\Gamma)$. The left-regular action of $\Gamma$, specified by $(\gamma f)(v) := f(\gamma^{-1}v)$, is a proper and free action by isometries of $\Gamma$ on the contractible space $\mathcal{S}(\Gamma)$, so that $\mathcal{S}(\Gamma)/\Gamma$ is a classifying space for $\Gamma$.

Fix a homology class $h \in H_k(\Gamma, \Z)$, let $\mathscr{C}(h)$ denote the collection of $k$-dimensional integral currents in the sense of \cite{ambrosio2000currents} representing $h$, and denote the mass of a current $C \in \mathscr{C}(h)$ by $\mathbf{M}(C)$. Define the spherical volume $\text{SphereVol}(h)$ of $h$ to be $\inf\limits_{C \in \mathscr{C}(h)} \mathbf{M}(C)$.

We remark that in the case where $h$ corresponds to the fundamental class $[\mathbb{1}_M]$ of a closed oriented Riemannian manifold $(M, g)$, this definition is equivalent to the spherical volume of Besson, Courtois, and Gallot \cite{besson1991entropie} (\cite{song2022plateau}, Remark 2.6). The essential observation here is that the setup of \cite{besson1991entropie} can be phrased independently of reference geometric data: if $D_M$ is a fundamental domain for the universal cover $(\widetilde{{M}}, \widetilde{g})$ of $(M,g)$, then $L^2(\widetilde{{M}}, d\text{Vol}_{\widetilde{g}}) \cong \ell^2(\pi_1(M), L^2(D_M, d\text{Vol}_{\widetilde{g}}))$ and $\ell^2(\pi_1(M), \mathbf{H})$ are $\pi_1(M)$-equivariantly isometrically identified. 

A solution to the spherical Plateau problem is defined to be an integral current space $C_\infty$ that is the intrinsic-flat limit (see \cite{sormani2011intrinsic}) of a sequence $C_i \in \mathscr{C}(h)$ so that $\lim\limits_{i\to\infty}\mathbf{M}(C_i) = \text{SphereVol}(h)$.

Part of the data of an integral current space $C$ is a complete metric space $(X, d)$, which for appropriate path-connected spaces $X$ induces a path metric $d^L$ on $X$. The main theorem we shall use from \cite{song2022plateau} is that if $(Y, {g}_{\text{hyp}})$ is a closed hyperbolic $3$-manifold, then for any Plateau problem solution $C_\infty$ corresponding to the fundamental class $[\mathbb{1}_Y]$, the path-metric $d_\infty^L$ associated to the metric space $(X_\infty, d_{\infty})$ is a metric on $X_\infty$ and $(X_\infty,\sqrt{3} d_\infty^L)$ is isometric to $(Y, g_{\text{hyp}})$ (Theorem 3.4). Song's proof of this is based on an analysis of barycenter maps associated to the currents $C_i$ converging to $C_\infty$, as bounds on the Jacobians of these barycenter maps grow sharp (c.f. \cite{besson1995rigidites}).
\section{Proofs}

The contribution of this note is the observation that because group homology classes analogous to fundamental classes of closed oriented hyperbolic $3$-manifolds exist for all torsion-free hyperbolic groups $\Gamma$ satisfying the hypotheses of Cannon's conjecture, Song's work implies that the Plateau problem solutions for these homology classes coincide up to homothety with the closed oriented hyperbolic $3$-manifold $Y$ with $\pi_1(Y) \cong \Gamma$ whenever such a manifold exists.

To begin, we compute the group homology $H_3(\Gamma, \mathbb{Z})$ for a group $\Gamma$ satisfying the hypotheses of Theorem \ref{headline-result}.

\begin{lemma}\label{homology-computation}
Let $\Gamma$ be a torsion-free hyperbolic group with Gromov boundary $S^2$. Then $H_3(\Gamma, \mathbb{Z}) \cong \mathbb{Z}$.
\end{lemma}
\begin{proof}
By Corollary 1.3(c) of \cite{bestvina1991groups}, since $\partial \Gamma$ is $S^2$, $\Gamma$ is an integral Poincar\'e duality group of dimension 3. The rest is a standard computation: $H_0(\Gamma, \mathbb{Z}) \cong \mathbb{Z}$ so by duality $H^3(\Gamma, \mathbb{Z}) \cong \mathbb{Z}$, and by the Universal Coefficient Theorem the top homology group is $\mathbb{Z}$ as well. 
\end{proof}

We now establish our result.

\begin{proof}[Proof of Theorem \ref{headline-result}.]
Let $\Gamma$ be a group satisfying the hypotheses of Theorem \ref{headline-result}. Lemma \ref{homology-computation} ensures that $H_3(\Gamma, \Z) \cong \Z$; let $h$ be a generator of $H_3(\Gamma,\Z)$ and let $C_\infty$ be a spherical Plateau problem solution for $h$.

Suppose that $(Y, g_{\text{hyp}})$ is a closed orientable hyperbolic $3$-manifold with $\pi_1(Y) \cong \Gamma$. As $Y$ is aspherical, its homology is isomorphic to that of $\mathcal{S}(\Gamma)/\Gamma$. Choose an orientation on $Y$ so that the isomorphism $\Gamma \cong \pi_1(Y)$ induces an equivariant isometry $\mathcal{S}(\Gamma) \to \mathcal{S}(\pi_1(Y))$ whose induced map on $H_3(\mathcal{S}(\Gamma)/\Gamma, \Z)$ maps $h$ to the class corresponding to $[\mathbb{1}_Y]$. Now apply Theorem 3.4 in \cite{song2022plateau}. The other implication is trivial.
\end{proof}

\bibliography{main}

\begin{bibdiv}
\begin{biblist}

\bib{ambrosio2018hilbertreg}{article}{
      author={Ambrosio, Luigi},
      author={De~Lellis, Camillo},
      author={Schmidt, Thomas},
       title={Partial regularity for mass-minimizing currents in {H}ilbert
  spaces},
        date={2018},
     journal={J. Reine Angew. Math.},
      volume={734},
       pages={99\ndash 144},
}

\bib{ambrosio2000currents}{article}{
      author={Ambrosio, Luigi},
      author={Kirchheim, Bernd},
       title={Currents in metric spaces},
        date={2000},
     journal={Acta Math.},
      volume={185},
       pages={1\ndash 80},
}

\bib{almgren1983codimension2}{article}{
      author={Almgren, Frederick~J., Jr.},
       title={{$Q$} valued functions minimizing {D}irichlet's integral and the
  regularity of area minimizing rectifiable currents up to codimension two},
        date={1983},
     journal={Bull. Amer. Math. Soc. (N.S.)},
      volume={8},
       pages={327\ndash 328},
}

\bib{ambrosio2016survey}{article}{
      author={Ambrosio, Luigi},
       title={The regularity theory of area-minimizing integral currents [after
  {A}lmgren--{D}e {L}ellis--{S}padaro]},
        date={2016},
     journal={Ast\'{e}risque},
      number={380, S\'{e}minaire Bourbaki. Vol. 2014/2015},
       pages={Exp. No. 1093, 139\ndash 169},
}

\bib{besson1991entropie}{article}{
      author={Besson, G\'erard},
      author={Courtois, Gilles},
      author={Gallot, Sylvestre},
       title={Volume et entropie minimale des espaces localement
  sym\'{e}triques},
        date={1991},
     journal={Invent. Math.},
      volume={103},
       pages={417\ndash 445},
}

\bib{besson1995rigidites}{article}{
      author={Besson, G\'erard},
      author={Courtois, Gilles},
      author={Gallot, Sylvestre},
       title={Entropies et rigidit\'{e}s des espaces localement sym\'{e}triques
  de courbure strictement n\'{e}gative},
        date={1995},
     journal={Geom. Funct. Anal.},
      volume={5},
       pages={731\ndash 799},
}

\bib{besson1996rigidites}{article}{
      author={Besson, G\'erard},
      author={Courtois, Gilles},
      author={Gallot, Sylvestre},
       title={Minimal entropy and {M}ostow's rigidity theorems},
        date={1996},
     journal={Ergodic Theory Dynam. Systems},
      volume={16},
       pages={623\ndash 649},
}

\bib{bonkkleiner2005}{article}{
      author={Bonk, Mario},
      author={Kleiner, Bruce},
       title={Conformal dimension and {G}romov hyperbolic groups with
  2–sphere boundary},
        date={2005},
     journal={Geom. Topol.},
       pages={219\ndash 246},
}

\bib{bartels2010hyperbolic}{article}{
      author={Bartels, Arthur},
      author={L{\"u}ck, Wolfgang},
      author={Weinberger, Shmuel},
       title={On hyperbolic groups with spheres as boundary},
        date={2010},
     journal={J. Differential Geom.},
      volume={86},
       pages={1\ndash 16},
}

\bib{bestvina1991groups}{article}{
      author={Bestvina, Mladen},
      author={Mess, Geoffrey},
       title={The boundary of negatively curved groups},
        date={1991},
     journal={J. Amer. Math. Soc.},
       pages={469\ndash 481},
}

\bib{cannon1998conjecture}{article}{
      author={Cannon, James~W.},
      author={Swenson, Eric~L.},
       title={Recognizing constant curvature discrete groups in dimension
  {$3$}},
        date={1998},
     journal={Trans. Amer. Math. Soc.},
       pages={809\ndash 849},
}

\bib{gabai2009minimum}{article}{
      author={Gabai, David},
      author={Meyerhoff, Robert},
      author={Milley, Peter},
       title={Minimum volume cusped hyperbolic three-manifolds},
        date={2009},
     journal={J. Amer. Math. Soc.},
      volume={22},
       pages={1157\ndash 1215},
}

\bib{hass2983slice}{article}{
      author={Hass, Joel},
       title={The geometry of the slice-ribbon problem},
        date={1983},
     journal={Math. Proc. Cambridge Philos. Soc.},
      volume={94},
       pages={101\ndash 108},
}

\bib{Markovic2013criterion}{article}{
      author={Markovi{\'c}, Vladimir},
       title={Criterion for {C}annon's conjecture},
        date={2013},
     journal={Geom. Funct. Anal.},
      volume={23},
       pages={1035\ndash 1061},
}

\bib{meeks1982plateau}{article}{
      author={Meeks, William~H., III},
      author={Yau, Shing~Tung},
       title={The classical {P}lateau problem and the topology of
  three-dimensional manifolds. {T}he embedding of the solution given by
  {D}ouglas-{M}orrey and an analytic proof of {D}ehn's lemma},
        date={1982},
     journal={Topology},
       pages={409\ndash 442},
}

\bib{perelman2002entropy}{misc}{
      author={Perelman, Grisha},
       title={The entropy formula for the {R}icci flow and its geometric
  applications},
   publisher={\href{https://arxiv.org/abs/math/0211159}{arXiv:math/0211159}},
        date={2002},
         url={https://arxiv.org/abs/math/0211159},
}

\bib{perelman2003extinction}{misc}{
      author={Perelman, Grisha},
       title={Finite extinction time for the solutions to the {R}icci flow on
  certain three-manifolds},
   publisher={\href{https://arxiv.org/abs/math/0307245}{arXiv:math/0307245}},
        date={2003},
         url={https://arxiv.org/abs/math/0307245},
}

\bib{perelman2003surgery}{misc}{
      author={Perelman, Grisha},
       title={Ricci flow with surgery on three-manifolds},
   publisher={\href{https://arxiv.org/abs/math/0303109}{arXiv:math/0303109}},
        date={2003},
         url={https://arxiv.org/abs/math/0303109},
}

\bib{song2022email}{article}{
      author={Song, Antoine},
       title={Candidates for {C}annon's conjecture},
        date={2022},
     journal={private communication},
}

\bib{song2022plateau}{misc}{
      author={Song, Antoine},
       title={The spherical {P}lateau problem for group homology},
   publisher={\href{https://arxiv.org/abs/2202.10636}{arXiv:math/2202.10636}},
        date={2022},
         url={https://arxiv.org/abs/2202.10636},
}

\bib{sormani2011intrinsic}{article}{
      author={Sormani, Christina},
      author={Wenger, Stefan},
       title={The intrinsic flat distance between {R}iemannian manifolds and
  other integral current spaces},
        date={2011},
     journal={J. Differential Geom.},
      volume={87},
       pages={117\ndash 199},
}

\bib{Thurston1982Geometrization}{article}{
      author={Thurston, William},
       title={Three-dimensional manifolds, {K}leinian groups and hyperbolic
  geometry},
        date={1982},
     journal={Bull. Amer. Math. Soc. (N.S.)},
       pages={357\ndash 381},
}

\bib{wenger2011compactness}{article}{
      author={Wenger, Stefan},
       title={Compactness for manifolds and integral currents with bounded
  diameter and volume},
        date={2011},
     journal={Calc. Var. Partial Differential Equations},
      volume={40},
       pages={423\ndash 448},
}

\end{biblist}
\end{bibdiv}

\bigbreak\noindent Department of Mathematics,
\smallbreak\noindent Rice University,
\smallbreak\noindent Houston, TX, 77005,
\smallbreak\noindent Email: tcw@rice.edu, adn5@rice.edu
\end{document}